\documentclass[]{article}

\usepackage{amsmath}         
\usepackage{amssymb} 
\usepackage{amsthm}  
\usepackage{enumerate}
\usepackage{stmaryrd}
\usepackage{url} 
\usepackage{mathrsfs}       
\usepackage[all]{xy}
\usepackage{graphicx}

\usepackage{a4wide}

\theoremstyle{plain}
\newtheorem{theorem}{Theorem}[]
\newtheorem{lemma}[theorem]{Lemma}
\newtheorem{corollary}[theorem]{Corollary}

\theoremstyle{definition}

\def\c #1{{\cal #1}}

\def\ms#1{\mathscr{#1}}             

\def\b#1{\mathbf{#1}}

\def\BBQ{\mathbb{Q}}

\def\Q{\mathbb{Q}}

\def\R{\mathbb{R}}
\def\Z{\mathbb{Z}}

\def\sc{{\mathsf c}}      %
\def\Var{\mathsf{Var}}

\def\La{\c L_{\rm a}}
\def\Lp{\c L^{+}}
\def\M{\mathcal{M}}

\def\SS{\mathcal{S}}

\def\cL{\c L}

\def\Prop{\mathit{Prop}}

\def\el{\mathrm{L}}

\def\bcap{\mathop{\textstyle\bigcap}}
\def\bcup{\mathop{\textstyle\bigcup}}

\def\bscap{\mathop{\textstyle\bigsqcap}}
\def\bscup{\mathop{\textstyle\bigsqcup}}

\def\ab #1{|#1|}
\def\abm#1{\ab{#1}^\M}
\def\abr#1{|#1|_\R}
\def\abs#1{\ab{#1}^{\M^*}}

\def\bicon{\leftrightarrow}

\def\dsh{\mathord{-}}           

\def\emp{\emptyset}

\def\id{\approx}

\def\ov#1{\overline{#1}}

\def\pm{\prec^\M}

\def\sub{\subseteq}

\def\<{\langle}
\def\>{\rangle}

\def\rp{/\!/}      

\def\D{\Delta}

\def\om{\omega}
\def\ph{\varphi}

\def\th{\theta}

\begin{document}    

\title{\Large Strong completeness of a first-order temporal logic for real time}
\author{Robert Goldblatt\thanks{School of Mathematics and Statistics,
Victoria University of Wellington, New Zealand. 
sms.vuw.ac.nz/$\sim$rob}}
\maketitle
\noindent

\begin{abstract}
Propositional temporal logic over the real number time flow is finitely axiomatisable, but its first-order counterpart is not recursively axiomatisable.
We study the logic that combines the propositional axiomatisation with the usual axioms for first-order logic with identity, and develop an alternative ``admissible'' semantics for it,  showing that it is strongly complete for admissible models over the reals. By contrast there is no recursive axiomatisation of the first-order temporal logic of admissible models whose time flow is the integers, or any scattered linear ordering.
\end{abstract}



\section{Introduction}

A set of formulas is called  \emph{recursively axiomatisable} if it is the set of theorems of some deductive system that has a recursive set of axioms and a recursive set of proofs \cite[\S 7]{ ende:elem77}.
For instance,
if the flow of time is modelled by the linearly ordered set $(\R,<)$ of real numbers, then the resulting temporal logic of valid \emph{propositional} formulas is recursively axiomatisable. This was  shown by Robert Bull \cite{bull:alge68}, using finitely many axioms and inference rules. 

The situation of\emph{ first-order} temporal logic is quite different.
Dana Scott proved in the 1960's that the logic of first-order temporal models over the reals has in general no recursive axiomatisation. The proof applies when the time flow is any infinite Dedekind complete linear order, including the integers $\Z$ and the natural numbers $\om$.

It is natural to consider the recursively axiomatised logic $\el_\R$ obtained by combining Bull's propositional axioms  and rules with those of classical first-order logic with identity.  But, by Scott's result, this
$\el_\R$ is not complete for validity in  first-order temporal models over the reals, and the incompleteness cannot be overcome by adding any recursive set of further axioms to $\el_\R$. What model-theoretic interpretation can 
$\el_\R$ have, if any?  Is there some  notion of validity over real time, different to the standard one, with respect to which $\el_\R$ is complete?

We give a positive answer to this question by adapting the `admissible'' semantics for quantified modal logic developed in \cite{gold:gene06,gold:quan11}. This is based on the idea that, while a proposition can be identified with the set of worlds, or moments of time, at which it is true, not all sets of worlds/times need correspond in this way to any proposition. We use models $\M$ whose components include a time flow $T$, a universe $U$ of individuals, and a designated collection $\Prop$ of subsets of $T$, called the \emph{admissible propositions} of $\M$, from which the interpretations of formulas are to be selected. In particular $\M$ assigns to each sentence $\ph$ a truth set $\abm{\ph}$ of all points in $T$ at which $\ph$ is true, and these truth sets are required to be admissible.
 $\Prop$ is closed under the Boolean set operations that interpret the truth-functional connectives, and under further operations interpreting the temporal modalities. The partial ordering $\sub$ of set inclusion serves as an entailment relation between propositions. $\Prop$ is also used to give an alternative interpretation of the quantifiers which takes into account the admissibility of propositions but still  validates the classical axioms and rules for $\forall$ and $\exists$.

In a standard model, a sentence $\forall x\ph$ is treated as semantically equivalent to the conjunction of the sentences $\ph(a/x)$ for all $a\in U$. This makes the truth sets $\abm{\forall x\ph}$, $\abm{\ph(a/x)}$ satisfy
$$
 \abm{\forall x\ph}= \bcap_{}\{\abm{\ph(a/x)}:a\in U\}.
 $$ 
An admissible model has instead that
$$
 \abm{\forall x\ph}= \bscap_{}\{\abm{\ph(a/x)}:a\in U\},
 $$
 where $\bscap$ is an operation that produces the  \emph{greatest lower bound} of the $\abm{\ph(a/x)}$'s in the partially ordered set $(\Prop,\sub)$. This means that $\abm{\forall x\ph}$ is an admissible proposition that entails all of the                      $\abm{\ph(a/x)}$'s, and is the weakest such proposition to do so, i.e.\ it is entailed by any other admissible proposition that entails all of the  $\abm{\ph(a/x)}$'s. Our definition of ``model'' ensures that this greatest lower bound  exists and is admissible whenever all of the  $\abm{\ph(a/x)}$'s are admissible. The existential quantifier is treated dually as 
 $$
 \abm{\exists x\ph}= \bscup_{}\{\abm{\ph(a/x)}:a\in U\},
 $$
 where the operation $\bscup$ produces least upper bounds in $\Prop$. The interpretation of $\forall$ and $\exists$ by greatest-lower and least-upper bounds has a long history that is briefly discussed in \cite[\S1.4]{gold:quan11}.
 
 The next two sections set out the background theory of admissible semantics for first-order temporal logic. Then in Section 4 we prove that for a countable language, $\el_\R$ is strongly complete over the set of admissible models based on $\R$. This is done by taking an $\el_\R$-consistent set $\D$ of formulas and applying known results to obtain a standard model that satisfies $\Delta$ and is based on the rational number time flow $(\Q,<)$. This model is then extending to an admissible model based on $\R$ that still satisfies $\D$. 
 
 The final section shows that this construction cannot be successfully carried out when the temporal order  is discrete. We prove that
the set of formulas valid in all admissible models over the \emph{integer} time flow $\Z$ is not recursively axiomatisable, just as for the logic of standard models over $\Z$. The non-axiomatisability argument extends to hold for any logic characterised by the admissible models over a time flow that is \emph{scattered}, i.e.\ does not contain a copy of $\Q$ as a suborder.

\section{Logics}

We review the syntax of first-order temporal logic.
Fix a denumerable set $\Var=\{x,y,z,\dots\}$ of individual variables and a  signature $\cL$, consisting of various individual constants $\sc$, function symbols $F$ and predicate symbols $P$.
An $\cL$-\emph{term} is any individual variable, any constant $\sc$ from $\cL$, or inductively any expression $F\tau_1\cdots\tau_n$ where $F$ is an $n$-ary function symbol from $\cL$, and $ \tau_1,\dots,\tau_n$ are $\c L$-terms. 
 A term is \emph{closed} if it contains no variables.
 
An \emph{atomic $\cL$-formula} is any expression $P\tau_1\cdots\tau_n$ or $\tau_1\id\tau_2$ where $P$ is an $n$-ary predicate symbol from $\cL$, and the $ \tau_i$ are $\c L$-terms. 
The set of $\cL$-\emph{formulas} is generated from the atomic ones in the usual way, using the Boolean connectives $\land$ (conjunction) and  $\neg$ (negation), the universal quantifiers $\forall x$ for each individual variable $x$, and the temporal modalities $\b G$ (`it will always be that') and  $\b H$ (`it has always been that').
Other Boolean connectives ($\lor$,  $\to$, $\bicon$),  and the existential quantifiers $\exists x$ are introduced by standard definitions. $\b F$ (`at some future time') and $\b P$ (`at some past time') are defined as $\neg\b G\neg$ and $\neg\b H\neg$
respectively. The modality $\Box$ is introduced by defining $\Box \ph$ to be the formula $\b H\ph\land\ph\land\b G\ph$. On a linear time flow $\Box $ is the universal modality expressing `at all times, past present and future'. $\lozenge\ph$ abbreviates $\b P\ph\lor\ph\lor\b F\ph$ (`at some time').
The notation $\ph(\tau/x)$ denotes the formula obtained by replacing all free occurrences of $x$ in formula $\ph$ by the term $\tau$. 
  Each formula or inference rule has a \emph{mirror image}, obtained by replacing $\b G$ by $\b H$ and vice versa, hence also interchanging $\b F$ and $\b P$.

We now list axioms and rules of inference that define the notion of a temporal logic that we will work with. The inference rules needed are

\medskip
\begin{tabular}{ll}
$\dfrac{\ph,\ \ph\to\psi}{\psi}$  &\qquad \emph{Modus Ponens}
\\
\noalign{\medskip}
$\dfrac{\ph}{\b G\ph}$ ,\quad  $\dfrac{\ph}{\b H\ph}$   &\qquad    \emph{Temporal Generalisation}
\\
\noalign{\medskip}
$\dfrac{\ph}{\forall x\ph}$      &\qquad \emph{Universal Generalisation}
\end{tabular}

\bigskip\noindent
We  adopt the following axioms about quantifiers.
 
\medskip
\begin{tabular}{ll}
$\forall x\ph\to \ph(\tau/x)$, \qquad where $\tau$ is free for $x$ in $\ph$.
&\quad \quad \emph{Universal Instantiation} 
\\
$\forall x(\ph\to \psi) \to (\forall x\ph\to\forall x\psi)$  
&\quad\quad \emph{Universal Distribution} 
\\
$\ph\to\forall x\ph$, \qquad where $x$ is not free in $\ph$.
&\quad\quad \emph{Vacuous Quantification} 
\end{tabular}

\bigskip\noindent
For the identity symbol $\id$  we need the following axioms, in which  the notation $\ph(\tau' \rp \tau)$ denotes any formula obtained from $\ph$ by replacing some, but not necessarily all, occurrences of $\tau$ by $\tau'$.

\medskip
\begin{tabular}{ll}
$\tau\id\tau$
&\quad\quad \emph{Self Identity} 
\\
$\tau\id\tau'\to(\ph\to\ph(\tau'\rp \tau))$, \quad with $\ph$ atomic.
&\quad \quad\emph{Substitution of Identicals} 
\\
$\tau\id\tau'\to \Box(\tau\id\tau')$
&\quad \quad\emph{Rigid Identity} 
\end{tabular}

\bigskip\noindent
By a \emph{(temporal) logic} over a signature $\c L$ we will mean any set $\el$ of $\c L$-formulas that is closed under the above rules and includes all instances of truth-functional tautologies and the above axioms, as well as the following two schemes and their mirror images.

\medskip
\begin{tabular}{ll}
$\b G(\ph\to\psi)\to(\b G\ph\to\b G\psi)$ &\quad 
\\
\noalign{\medskip}
$\ph\to\b G\b P\ph$    &\qquad\qquad   
\end{tabular}

\medskip\noindent
Members of $\el$ are called \emph{$\el$-theorems}. The last two schemes and their mirror images axiomatise the propositional logic K$_t$ due to E.~J.~Lemmon and known as the ``minimal tense-logic" \cite[p.176]{prio:past67}.

It is significant that the \emph{Barcan formulas} 
\begin{equation}\label{barcan}
\forall x\b G\ph\to\b G\forall x\ph, \qquad  \forall x\b H\ph\to\b H\forall x\ph
\end{equation}
are derivable as theorems of any logic as defined here. The derivation depends on Universal Instantiation as well as the scheme $\ph\to\b G\b P\ph$ and its mirror: see \cite[p.147]{prio:past67} or \cite[Theorem 11.5]{gabb:inve76}.

The logic $\el_\Q$ over $\c L$ is defined to be the smallest  logic that includes the following schemes and their mirror images (with the properties of the temporal order that they encapsulate listed on the right):

\medskip
\begin{tabular}{ll}
\noalign{\medskip}
$\b G\ph\to\b G\b G\ph$     &\qquad \emph{transitivity}
\\
$\b F\ph\land\b F\psi\to \b F(\ph\land\psi)\lor \b F(\ph\land\b F\psi)\lor\b F(\b F\ph\land\psi)$      &\qquad \emph{linear future}
\\
\noalign{\medskip}
$\b G\ph\to\b F\ph$      &\qquad \emph{endless future }
\\
\noalign{\medskip}
$\b G\b G\ph\to\b G\ph$    &\qquad \emph{density}
\end{tabular}

\medskip\noindent
Thus the $\el_\Q$-theorems over $\c L$ are all $\c L$-formulas that are obtainable from the above axioms by applying the given rules. 
The last six schemes together axiomatise the propositional temporal logic of $(\Q,<)$ \cite{bull:alge68}.

\medskip\noindent
The logic $\el_\R$ can be defined by adding to $\el_\Q$ the axiom
\begin{equation}\label{cty}
\Box(\b G\ph\to\b P\b G\ph)\to(\b G\ph\to\b H\ph) \qquad\qquad\qquad\text{\emph{Dedekind completeness}}
\end{equation}

For a given logic $\el$, a formula $\ph$ is \emph{$\el$-consistent} if $\neg\ph$ is not an $\el$-theorem, and a set $\D$ of formulas is $\el$-consistent if every finite subset of $\D$ has an $\el$-consistent conjunction. An $\el$-consistent set $\D$ of $\c L$-formulas is \emph{maximally $\el$-consistent} over $\c L$ if there is no $\el$-consistent set of $\c L$-formulas that properly extends it, or equivalently if it is \emph{negation complete} in the sense that for any $\c L$-formula $\ph$, if  $\ph\notin\D$ then $\neg\ph\in\D$.  A maximally $\el$-consistent set contains all  $\el$-theorems.

Write $\D\vdash_\el\ph$ to mean that there is a formula $\delta\to\ph$ in $\el$ such that $\delta$ is the conjunction of finitely many members of $\D$.
$\D$ is \emph{$\forall$-complete} over $\c L$ if for every $\c L$-formula $\ph$ and variable $x$, if $\D\vdash_\el\ph(\tau/x)$ for all \emph{closed} $\c L$-terms $\tau$, then  $\D\vdash_\el\forall x\ph$.  $\D$ is \emph{$\el$-saturated over $\c L$} if it is both
maximally $\el$-consistent and $\forall$-complete over $\c L$.
The following is a well-known result that goes back to \cite{henk:gene57}.

\begin{theorem}  \label{henkin}
For any logic $\el$ over a countable signature $\c L$,
if $\D$ is $\el$-consistent over $\c L$, and $\c L^+$ is a countable extension of $\c L$ that includes infinitely many constants not in $\c L$, then  $\D$ can be extended to an $\el$-saturated set over $\c L^+$.
\hfill\qed
\end{theorem}

.

\section{Admissible Models}

A \emph{time flow} is a structure $(T,<)$ for which  $T$ is a set (thought of as a set of times/instants), and $<$ is a linear ordering on $T$, i.e.\ a transitive irreflexive relation that is linear in the sense that either $s<t$ or $t<s$ for all distinct $s,t\in T$. The relation inverse to $<$ will be denoted $>$.
When $T$ is a set of numbers, e.g.\ the reals $\R$, rationals $\Q$, integers $\Z$ or natural numbers $\om$, then $<$ will invariably have its usual numerical meaning.

A \emph{model structure} is a system
$\SS=(T,<,\Prop,U)$
such that $(T,<)$ is a time flow;
$\Prop$ is a non-empty subset of the powerset $\wp T$ of $T$ that is closed under the Boolean set operations and under
 the operations $[<]$ and $[>]$, interpreting $\b G$ and $\b H$, defined by
\begin{align*}
[<]X&=\{t\in T: \forall s\in T(t<s\text{ implies } s\in X)\},
\\
[>]X&=\{t\in T: \forall s\in T(t>s\text{ implies } s\in X)\};   
\end{align*}
and $U$ is a set, called the \emph{universe} of $\SS$. A subset of $T$ is called \emph{admissible} if it belongs to $\Prop$.
Members of $\Prop$ may be referred to as the admissible propositions of $\SS$.

Operations $\bscap$ and $\bscup$  on collections of subsets of $T$ are defined by putting, for each $\c Z\sub\wp T$,
\begin{align*}
\bscap \c Z &= \bcup\{Y\in\Prop:Y\sub \bcap \c Z\},
\\
  \bscup \c Z &= \bcap\{Y\in\Prop: \bcup \c Z\sub Y\}.
\end{align*}
We emphasise that $\bscap\c Z$ and $\bscup\c Z$ are defined for arbitrary $\c Z\sub\wp W$. They need not be admissible, even when all members of $\c Z$ are admissible. If $\c Z\sub\Prop$ and $\bscap Z$ is admissible, then $\bscap\c Z$ is the greatest lower bound of $\c Z$ in the partially ordered set $(\Prop,\sub)$, and will not be equal to $\bcap\c Z$ unless the latter is admissible, which it need not be.  Dual statements hold about $ \bscup \c Z$ and $\bcup \c Z$ and least upper bounds \cite[\S 1.5]{gold:quan11}.

Another perspective is that $\Prop$ is a base for a topology on $T$ for which  $ \bscap \c Z$ is the interior of $ \bcap \c Z$, while  $\bscup \c Z$ is the closure of $\bcup \c Z$.

A \emph{premodel} $\M=(\SS,\abm{\dsh})$ for $\c L$ based on 
$\SS$ is given by an interpretation function   
$\ab{\dsh}^\M$ that assigns
to each individual constant $\sf c\in\cL$ an element $\ab{{\sf c}}^\M$ of the universe $ U$; 
to each $n$-ary function symbol $F\in\cL$ an $n$-ary function 
$\ab{F}^\M:U^n\to U$;
and
to each $n$-ary predicate symbol $P\in\cL$ a function $\ab{P}^\M:U^n\to\wp W$.
Informally $\ab{P}^\M(a_1,\dots,a_n)$  represents  the proposition that the predicate $P$ holds of the $n$-tuple $(a_1,\dots,a_n)$. Equivalent ways of interpreting $P$ are to assign to it the function $t\mapsto\{(a_1,\dots,a_n):t\in \ab{P}^\M(a_1,\dots,a_n)\}$ from $T$ to $U^n$, or  the subset $\{(t,a_1,\dots,a_n):t\in \ab{P}^\M(a_1,\dots,a_n)$\} of $T\times U^n$.

A \emph{variable assignment} is a function $f$ assigning to each $x\in\Var$ an element $fx$ of the universe $U$. Then $f$ inductively assigns
 to each $\cL$-term $\tau$ a value $\abm{\tau}f\in U$, by  putting $\abm{x}f=fx$; $\abm{\sc}f=\abm{\sc}$;
and
$\abm{F\tau_1\cdots\tau_n}f = \ab{F}^\M(\abm{\tau_1} f,\dots,\abm{\tau_n} f$). If $\tau$ is a \emph{closed} term, then $\abm{\tau}f=\abm{\tau}g$ for all variable assignments $f,g$, and we denote this constant value by $\abm{\tau}$.
We use the notation $f[a/x]$ for the variable assignment that \emph{updates}  $f$ by assigning the value $a$ to $x$ and otherwise acting identically to  $f$.

A premodel $\M$ associates with each $\c L$-formula $\ph$ and  assignment $f$ a subset $|\varphi|^\M  f$ of $T$, viewed as the  ``truth set'' of all times at which $\varphi$ is true under $f$. This is defined by induction on the formation of $\ph$:

\begin{itemize}
\item
$\abm{\tau_1\id\tau_2}f$ is $T$ if $\abm{\tau_1}f=\abm{\tau_2}f$, and is $\emp$ otherwise.
\item 
$\ab{P\tau_1\cdots\tau_n}^\M f=\ab{P}^\M(\abm{\tau_1} f,\dots,\abm{\tau_n} f)$.
\item
$|\varphi\wedge\psi|^\M f=|\varphi|^\M f\cap|\psi|^\M f$.
\item 
$|\neg\varphi|^\M f=T-|\varphi|^\M f$.
\item 
$|\b G\varphi|^\M f=[<]|\varphi|^\M f$.
\item 
$|\b H\varphi|^\M f=[>]|\varphi|^\M f$.
\item
$|\forall x\varphi|^\M  f = \bscap_{a\in U} |\varphi|^\M  f[a/x] $.
\end{itemize}
Since $\exists x\ph$ is $\neg\forall x\neg\ph$, it follows that
\begin{itemize}
\item
$| \exists x\varphi|^\M  f = \bscup_{a\in U}|\varphi|^\M  f[a/x]$.
\end{itemize}
Writing $\M,t,f\models\ph$ to mean that $t\in |\varphi|^\M  f$, we have the following clauses for this truth/satisfaction relation:
\begin{itemize}
\item
$\M,t,f\models\tau_1\id\tau_2$  \enspace iff \enspace  $\abm{\tau_1}f=\abm{\tau_2}f$.

\item $\M,t,f\models P\tau_{1}\cdots\tau_{n}$ \enspace iff \enspace $t\in \ab{P}^\M(\abm{\tau_1} f,\dots,\abm{\tau_n} f)$.

\item
$\M,t,f\models\ph\land\psi$ \enspace iff \enspace$\M,t,f\models\ph$ and $\M,t,f\models\psi$.

\item
$\M,t,f\models\neg\ph$ \enspace iff \enspace$\M,t,f\not\models\ph$.

\item $\M,t,f\models\b G \ph$ \enspace iff \enspace for all $s\in T,\ t<s\ \text{implies}\  \M,s,f\models
\ph$.

\item $\M,t,f\models\b H \ph$ \enspace iff \enspace for all $s\in T,\ s<t\ \text{implies}\  \M,s,f\models
\ph$.

\item $\M,t,f\models\b F \ph$ \enspace iff \enspace for some $s\in T,\ t<s\ \text{and}\  \M,s,f\models
\ph$.

\item $\M,t,f\models\b P \ph$ \enspace iff \enspace for some $s\in T,\ s<t\ \text{and}\  \M,s,f\models
\ph$.

\item $\M,t,f\models\Box \ph$ \enspace iff \enspace for all $s\in T,\   \M,s,f\models \ph$.

\item $\M,t,f\models\lozenge \ph$ \enspace iff \enspace for some $s\in T,\   \M,s,f\models \ph$.

\item $\M,t,f\models\forall x\ph$ \enspace iff \enspace there is an $X\in\Prop$ such that $t\in X \sub
\bcap_{a\in U}|\ph|^\M{f[a/x]}$.\label{pagesemantsforall}
\end{itemize}
From the clause for $\forall$  we see that
\begin{equation}\label{allkrip}
\text{
$\M,t,f\models\forall x\ph$ \enspace \emph{only if} \enspace for all $a\in U$, \ $\M,t,f[a/x]\models\ph$.}
\end{equation}
The converse need not hold.

Reading off the semantics for the defined existential quantifier gives
\begin{itemize}
\item[]
$\M,t,f\models\exists x\ph$  iff \enspace for all $X\in\Prop$ such that $t\in X$, there exists 
$s\in X$ and $a\in U$

\phantom{$\M,w,f\models \exists x\ph$\enspace iff}
 with $\M,s,f[a/x]\models\ph$.
\end{itemize}
Consequently:
\begin{equation}\label{somekrip}
\text{
If some $a\in U$ has  $\M,t,f[a/x]\models\ph$, then $\M,t,f\models\exists x\ph$.}
\end{equation}
Again, the converse can fail.

The semantics of term substitution is given by the  result \cite[1.6.2]{gold:quan11} that
\begin{equation} \label{termsub}
\text{$\M,t,f\models\ph(\tau/x)$ \enspace iff \enspace  $\M,t,f[\,\abm{\tau} f/x]\models\ph$.}
\end{equation}
It can also be shown  \cite[1.6.1]{gold:quan11} that if two assignments $f,g$ agree on all free variables of $\ph$, then
$\abm{\ph}f=\abm{\ph}g$. Hence if $\ph$ is a \emph{sentence} (no free variables), then $\abm{\ph}f=\abm{\ph}g$ for all assignments $f,g$  and  we write this single truth set as $\abm{\ph}$.
Also if $\ph$ is a sentence and
$\M,t,f\models\ph$, then $\M,t,g\models\ph$
for all assignments $g$, which we just write as $\M,t\models\ph$. Then $\abm{\ph}=\{t\in T:\M,t\models\ph\}$.

A formula $\ph$ is \emph{valid in} premodel $\M$, written $\M\models\ph$, if $\ab{\ph}^\M f=T$ for all $f$, i.e.\  if $\M,t,f\models\ph$  for all $t\in T$ and all variable assignments $f$ in $\M$.
$\ph$ is \emph{admissible in} premodel $\M$ if $\abm{\ph}f\in\Prop$ for all variable assignments $f$ in $\M$. A \emph{model} for $\c L$ is, by definition, an $\c L$-premodel in which every $\c L$-formula is admissible. We sometimes call this an \emph{admissible model} for emphasis or contrast.
 
It follows from \cite[\S 1.7]{gold:quan11} that the Universal Instantiation and Universal Distribution axioms are valid in any premodel $\M$ (Instantiation depends on \eqref{termsub}), while the Universal Generalisation rule is  sound for validity in $\M$. Moreover an instance $\ph\to\forall x\ph$ of Vacuous Quantification is valid in $\M$ provided that $\ph$ is admissible in $\M$. Thus if $\M$ is  a {model},  it validates Vacuous Quantification.

Validity of the Identity axioms in all premodels follows readily from the fact that the interpretation of $\id$ is  \emph{rigid}, i.e.\ independent of time. Under a given  variable assignment $f$, an identity
 $\tau_1\id\tau_2$ is true at some time iff it is true at all times, as its truth is determined  by the identity of the time-independent values $\abm{\tau_i}f$.

A logic  $\el$ is \emph{sound} for validity in a class  $\ms C$ of premodels if every $\el$-theorem is valid in all members of $\ms C$, or equivalently every formula that is satisfiable at a point in some member of $\ms C$ is $\el$-consistent. Conversely $\el$ is \emph{complete} over $\ms C$ if every $\el$-consistent formula is satisfiable in a member of $\ms C$, and is \emph{strongly complete} if every   $\el$-consistent set of formulas is satisfiable in a member of $\ms C$.
  
 Putting the above  observations about validity  together we see that the logic $\el_\Q$ is sound for validity in all \emph{models} based on the rational time flow. Likewise $\el_\R$ is validated by all models over the real time flow.

A premodel $\M$ will be called \emph{Kripkean} if  it always has
\begin{equation}\label{defkripkean}
|\forall x\varphi|^\M  f = \bigcap_{a\in U} |\varphi|^\M  f[a/x] .
\end{equation}
This means that $\forall$ gets the classical semantics
\begin{equation}\label{kripsem}
\text{
$\M,t,f\models\forall x\ph$\enspace iff \enspace for all $a\in U$, \ $\M,t,f[a/x]\models\ph$},
\end{equation}
and correspondingly, the existential quantifier gets
\begin{equation}\label{kripsemExists}
\text{
$\M,t,f\models\exists x\ph$\enspace iff \enspace for some $a\in U$, \ $\M,t,f[a/x]\models\ph$}
\end{equation}
(compare with \eqref{allkrip} and \eqref{somekrip}).

If a  premodel $\M$ has every subset admissible, i.e.\ $\Prop= \wp T$, then $\M$ is a model and $\bscap$ is just $\bcap$, so $\M$ is Kripkean. Such an $\M$ will be called a \emph{standard model}, and can be identified with the structure $(T,<,U,\abm{\dsh})$.  In \cite[\S\S 4.3--4.5]{gabb:temp94} a strong completeness theorem is proven for the logic  $\el_\Q$, showing that any $\el_\Q$-consistent set is satisfiable in a  standard model over the rational time flow. The method involves extending a consistent set to a saturated one over an enlarged signature, and then constructing a satisfying model for the saturated set. We distill from that work the following result.

\begin{theorem} \label{LQcomplete}
Let $\D$ be an $\el_\Q$-saturated set of $\c L$-formulas, where  $\c L$ is a countable signature having constant terms.  Then there is a standard $\c L$-model $\M$ over $(\Q,<)$ in which $\D$ is satisfied, i.e.\ $\M, q_0,f_0\models\D$ for some rational $q_0$ and some variable assignment $f_0$. Moreover each element of the universe of $\M$ is the value of some constant $\c L$-term.   \strut\hfill$\qed$
\end{theorem}

In a premodel $\M$ for which each element $a\in U$ is the value of some constant term $\bar{a}$, i.e.\ $\abm{\bar{a}}=a$,  any variable assignment $f:\Var\to U$ induces a substitution operator turning each formula $\ph$ into a \emph{sentence} $\ph^f$ by replacing each free variable $x$ of $\ph$ by the constant term $\ov{fx}$.
This allows us to reduce satisfaction of formulas by assignments to truth of sentences:  

\begin{lemma} \label{substf}
$\M,t,f\models\ph$ iff   $\M,t\models\ph^f$.
\end{lemma}   
\begin{proof}
By the result \eqref{termsub},  $\M,t,f\models\ph(\ov{fx}/x)$  iff
$\M,t,f[f(x)/x]\models\ph$.    
But  $f[f(x)/x]=f$.  Applying this to all the free variables of $\ph$  shows that 
$\M,t,f\models\ph^f$ iff   $\M,t,f\models\ph$.  But $\M,t,f\models\ph^f$ iff   $\M,t\models\ph^f$ since $\ph^f$ is a sentence. 
\end{proof}

A further consequence of having every element of the universe of $\M$ being a value $\abm{\bar{a}}$ is that if $\M$ is Kripkean and $\forall x\ph$ is a sentence, then by \eqref{kripsem} and \eqref{termsub},
\begin{equation}\label{kripall}
\text{
$\M,t\models\forall x\ph$\enspace iff \enspace for all $a\in U$, \ $\M,t\models\ph(\bar{a}/x)$.}
\end{equation}

\section{Strong completeness over the reals}

We will now prove that the logic $\el_\R$ over a countable signature $\c L$ is strongly complete over the set of admissible 
$\c L$-models based on $(\R,<)$.

 Let $\D$ be an $\el_\R$-consistent set of $\c L$-formulas. Our task is to show that it is satisfied in a some model over $(\R,<)$. Let $\Lp$ be the signature obtained by adding denumerably many new constants to $\c L$.
 By Theorem \ref{henkin} there is an $\el_\R$-saturated set $\D^+$ of $\Lp$-formulas with $\D\sub\D^+$.

 Now $\D^+$ is also  $\el_\Q$-saturated over $\Lp$: as well as being $\forall$-complete it is  $\el_\Q$-consistent  as $\el_\Q$  is a sublogic of $\el_\R$, and is negation complete for $\Lp$-formulas, hence is maximally $\el_\Q$-consistent over $\Lp$.
 So we can apply Theorem \ref{LQcomplete}  to $\D^+$ and $\Lp$ to conclude that $\D^+$ is satisfiable in a standard model   based on the rational time flow. So there is a standard $\Lp$-model
$$
\M=(\Q,<,U,\ab{\dsh}^\M)
$$
such that $\M, q_0,f_0\models\D^+$ for some rational $q_0$ and some variable assignment $f_0$. Also each $a\in U$ is $\abm{\bar a}$ for some closed $\Lp$-term $\bar a$.

\begin{lemma}  \label{LRM}
If a sentence $\ph$ is an $\el_\R$-theorem over $\Lp$, then $\M\models\ph$.
\end{lemma}
\begin{proof}
If $\ph$ is an $\el_\R$-theorem, then so is $\Box\ph$, which thus belongs to the $\el_\R$-saturated $\D^+$. That gives 
$\M,q_0\models\Box\ph$. Hence every rational $q$ has  $\M,q\models\ph$.
\end{proof}

We  will extend $\M$ to an  $\Lp$-model 
$$
\M^*=(\R,<,\Prop,U,\abs{\dsh}),
$$
based on the real time flow, with the same universe $U$ as $\M$, and with $\abs{\dsh}$ assigning the same interpretation as $\abm{\dsh}$ to the individual constants and function symbols of $\Lp$.
To define $\Prop$, and $\abs{P}$ for predicate symbols $P$, we will assign to each $r\in \R$ a certain set $\D_r$ of $\Lp$-\emph{sentences}.

An $\el_\R$-consistent set of $\Lp$-sentences will be called \emph{$\el_\R$-maximal} if there is no $\el_\R$-consistent set of      $\c\Lp$-sentences that properly extends it, or equivalently if it is \emph{negation complete for sentences}, i.e.\ it contains either $\ph$ or $\neg\ph$ for any $\Lp$-sentence $\ph$.  We define $\D_r$ to be $\el_\R$-maximal for all $r\in\R$.
 For $q\in\BBQ$, let   $\D_q$ be the set of all  $\Lp$-sentences $\ph$ such that $\M,q\models\ph$.

\begin{lemma}  \label{delmax}
For all $q\in \BBQ$, $\D_q$ is  $\el_\R$-maximal over $\Lp$.
\end{lemma}

\begin{proof}
 $\D_q$ is negation complete for $\Lp$-sentences, so it is enough to show that it is $\el_\R$-consistent. 
 Let $\ph$ be the conjunction of any finite number of members of $\D_q$. Then $\M,q\models\ph$, so $\M\not\models\neg\ph$.
  Hence by Lemma \ref{LRM}, $\neg\ph$ is not an $\el_\R$-theorem, i.e. $\ph$ is $\el_\R$-consistent as required.
 \end{proof}

For $r\notin\Q$, to define $\D_r$ we adapt a method used in \cite[p.190]{gabb:temp94}.
First define a set $\D_r^0$ to consist of each $\Lp$-sentence $\ph$ that is true in $\M$ throughout some open rational interval around $r$, i.e.\ those $\ph$ for which there exist $s,t$ with $s<r<t$ such that $\M,q\models\ph$ for all rationals $q$ having $s<q<t$.

\begin{lemma}   \label{Dprops}
A sentence $\psi$ belongs to $\D_r^0$ if it satisfies any of the following conditions.
\begin{enumerate}[\rm(1)]
\item
$\psi$ is $\b F\ph$ with  $\M,q\models\ph$  for some $q\in\BBQ$ having $r<q$.
\item
$\psi$ is $\b G\ph$ with  $\M,q\models\ph$  for all $q\in\BBQ$ having $r<q$.
\item
$\psi$ is $\b P\ph$ with  $\M,q\models\ph$  for some $q\in\BBQ$ having $q<r$.
\item
$\psi$ is $\b H\ph$ with  $\M,q\models\ph$  for all $q\in\BBQ$ having $q<r$.
\item
$\M\models\psi$.
\end{enumerate}
\end{lemma}

\begin{proof}
\begin{enumerate}[\rm(1)]
\item
If $\M,q\models\ph$ with $r<q$, then $\b F\ph$ is true in $\M$ throughout any rational interval around $r$ whose members are less than $q$, hence $\b F\ph\in\D_r^0$.

\item
Let $\M,q\models\ph$  for all rational $q$ greater than $r$. Thus the truth set $\abm{\ph}$ of $\ph$ in $\M$ includes
$\{q\in\BBQ:r<q\}$. Hence $\abm{\b G\ph}$  includes
$\{q\in\BBQ:r<q\}$.

Suppose, for the sake of contradiction, that  $\b G\ph\notin\D_r^0$. Then any open rational interval around $r$ must contain a point at which $\b G\ph$ is false in $\M$, and so this point must be less than $r$ (n.b.\ $r\notin\Q$). In particular, for any rational $q<r$, there must be a rational $q'$ with $q<q'<r$ and  
$\M,q'\not\models\b G\ph$, hence  $\M,q\not\models\b G\ph$. Altogether this shows that 
$\abm{\b G\ph}$ is exactly equal to $\{q\in\BBQ:r<q\}$.

Therefore by density of the rationals in $\R$, any point  in $\M$ at which $\b G\ph$ is true has $\b{PG}\ph$ true, so $\M\models\Box(\b G\ph\to\b{PG}\ph)$.
At the same time any rational $q>r$ has $\M,q\not\models \b G\ph \to \b H\ph$, so we now have  an instance of the $\el_\R$-axiom \eqref{cty} that is false in $\M$ at $q$. But that  contradicts Lemma \ref{LRM}.
The contradiction forces us to conclude that  $\b G\ph\in\D_r^0$.

\item
This is the mirror image of part (1).

\item
This is the mirror image of part (2), so can be proven using the mirror image of axiom \eqref{cty}. However it can also be shown from  
axiom \eqref{cty} itself, so we only need a single Dedekind completeness axiom.

To see this, suppose that  $\M,q\models\ph$  for all rational $q<r$, but $\b H\ph\notin\D_r^0$. Then reasoning dual to part (2), we get $\abm{\b H\ph}=\{q\in\BBQ:q<r\}$. Hence as $r\notin\BBQ$, 
$
\{q\in\BBQ:r<q\}=\abm{\neg\b H\ph}.
$
From this we see  that $\{q\in\BBQ:r<q\}=\abm{\b G\psi}$, where $\psi$ is $\neg\b H\ph$.
Hence as in the proof of part (2) we get that $\M\models\Box(\b G\psi\to\b{PG} \psi)$, while any rational $q>r$ has 
$\M,q\not\models \b G\psi \to \b H\psi$, so again we get the contradiction of  an instance of axiom \eqref{cty} falsifiable in $\M$.
\item
If $\M\models\psi$, then any open rational interval around $r$ has $\psi$ true throughout it.
\end{enumerate}
\end{proof}

\begin{lemma}
$\D_r^0$ is  $\el_\R$-consistent.
\end{lemma}

\begin{proof}
Let $\ph$ be the conjunction of any finite number of members of $\D_r^0$ . Each of these members is true in $\M$ throughout some rational interval around $r$. Choose a rational number $q$ that belongs to all of these finitely many intervals. Then $\M,q\models\ph$, so $\ph\in\D_q$. Hence $\ph$ is $\el_\R$-consistent as required, because $\D_q$ is 
 $\el_\R$-consistent by Lemma \ref{delmax}.
\end{proof}

If follows that  $\D_r^0$ has  $\el_\R$-maximal extensions. We choose one to be $\D_r$. That completes the definition of $\D_r$ for all reals $r$.

For each $\Lp$-sentence $\ph$ define $\abr{\ph} =\{r\in\R:\ph\in\D_r\}$, and put
$$
\Prop=\{\abr{\ph}: \ph \text{ is an $\Lp$-sentence}\}.
$$
 For an $n$-ary predicate symbol $P$, define $\abs{P}:U^n\to\Prop$ by
  $$
  \abs{P}(a_1,\dots,a_n)=\abr{P\ov{a_1}\cdots\ov{a_n}}.
  $$
That completes the definition of $\M^*$ as a \emph{premodel}.

\begin{lemma}  \label{phiR}
For any $\Lp$-sentence $\ph$,
$\M\models\ph$ iff $\abr{\ph}=\R$.
\end{lemma}

\begin{proof}
Let $\M\models\ph$. Then for any real $r$, if $r\in\BBQ$ we have $\M,r\models\ph$, so $\ph\in\D_r$. But if  $r\notin\BBQ$, then $\ph\in\D_r^0$ by Lemma \ref{Dprops}(5), so again $\ph\in\D_r$. In both cases $r\in\abr{\ph}$.

For the converse, let $\abr{\ph}=\R$. Then for any $q\in\Q$, $\ph\in\D_q$ and so $\M,q\models\ph$. Hence $\M\models\ph$.
\end{proof}

 Standard properties of maximally consistent sets ensure that
$$
\abr{\ph}\cap\abr{\psi}=\abr{\ph\land\psi}, \enspace
\abr{\ph}\cup\abr{\psi}=\abr{\ph\lor\psi}, \enspace
 \R-\abr{\ph}= \abr{\neg\ph}, 
$$
so $\Prop$ is a Boolean set algebra. It is also closed under the temporal operators $[<]$ and $[>]$, by

\begin{lemma}   \label{GHsem}
For any $\Lp$-sentence $\ph$,
$[<]\abr{\ph}=\abr{\b G\ph}$ and  $[>]\abr{\ph}=\abr{\b H\ph}$.
\end{lemma}

\begin{proof}
Let  $r\in [<]\abr{\ph}$. Then any $s$ with $r<s$ has $s\in\abr{\ph}$, i.e.\ $\ph\in\D_s$. In particular, any rational $q>r$ has
$\ph\in\D_q$, i.e. $\M,q\models\ph$. If $r\in\BBQ$, this implies $\M,r\models\b G\ph$,  so $\b G\ph\in \D_r$, hence $r\in \abr{\b G\ph}$. If $r\notin\BBQ$, it gives  $\b G\ph\in \D_r^0\sub\D_r$ by Lemma \ref{Dprops}(2), hence again $r\in \abr{\b G\ph}$. That proves
$[<]\abr{\ph}\sub\abr{\b G\ph}$.

For the converse inclusion, let $\b G\ph\in\D_r$. 
If $r\in\BBQ$, then $\M,r\models\b G\ph$, hence
\begin{equation}\label{ratcase}
\text{every rational $q$ greater than $r$ has $\M,q\models\ph$. }
\end{equation}
But if $r\notin\BBQ$, then since $\neg\b G\ph\notin\D_r$ and
 $\b F\neg\ph\to\neg\b G\phi$ is an $\el_\R$-theorem, we must have  $\b F\neg\ph\notin\D_r$, hence $\b F\neg\ph\notin\D_r^0$. This also gives \eqref{ratcase}, by Lemma \ref{Dprops}(1).
Now to show that  $r\in [<]\abr{\ph}$, take any $s$ with $r<s$. 
Then we need to show that $\ph\in\D_s$.
If $s\in\BBQ$, \eqref{ratcase} immediately gives $\M,s\models\ph$, hence $\ph\in\D_s$.
If $s\notin\BBQ$, take a rational $q_1$ with $r<q_1<s$. Then  \eqref{ratcase}  implies that every  rational $q$ greater than $q_1$ has  $\M,q\models\ph$. Hence $\M,{q_1}\models\b G\ph$. Therefore $\b{PG}\ph\in\D_s^0$ by Lemma \ref{Dprops}(3).
By the $\el_\R$-theorem $\b{PG}\ph\to\ph$,  this implies $\ph\in\D_s$ as required.

That concludes the proof of the first equation of the lemma. The proof of the second is its mirror image.
\end{proof}

\begin{lemma}  \label{allchar}
If $\forall x\ph$ is an $\Lp$-sentence, then in $\M^*$,\enspace
$\abr{\forall x\ph}=\bscap_{a\in U}\abr{\ph(\bar{a}/x)}$.
\end{lemma}

\begin{proof}
For any $r\in \R$, if $\forall x\ph\in\D_r$, then for all $a\in U$, by the  Universal Instantiation axiom  we get $\ph(\bar{a}/x)\in \D_r$, as maximal sets are closed under modus ponens.
Thus   
$\abr{\forall x\ph}\sub\bcap_{a\in U}\abr{\ph(\bar{a}/x)}$.
But  $\abr{\forall x\ph}\in\Prop$, so then $\abr{\forall x\ph}\sub\bscap_{a\in U}\abr{\ph(\bar{a}/x)}$.

Conversely, let $r\in \bscap_{a\in U}\abr{\ph(\bar{a}/x)}$. Then there is some $X\in\Prop$ with $r\in X$ and
\begin{equation} \label{Xsub}
X\sub \bcap_{a\in U}\abr{\ph(\bar{a}/x)}.
\end{equation}
Now  $X=\abr{\psi}$ for some sentence $\psi$. We show that $\M\models\psi\to\forall x\ph$. For if $\M,q\models\psi$, then $\psi\in\D_q$, hence
$q\in X$, so by \eqref{Xsub}, for all $a\in U$ we get $q\in\abr{\ph(\bar{a}/x)}$, so $\M,q\models\ph(\bar{a}/x)$.
Then $\M,q\models\forall x\ph$ as $\M$ is Kripkean \eqref{kripall}. This shows that $\M,q\models\psi\to\forall x\ph$ for all $q$ in $\BBQ$ as claimed.

By Lemma \ref{phiR} it follows that  $\abr{\psi\to\forall x\ph}=\R$. This yields $\abr{\psi}\sub \abr{\forall x\ph}$.  But $r\in\abr{\psi}$, so we get  $r\in \abr{\forall x\ph}$ as required to complete the proof.
\end{proof}

\begin{theorem}
For any $\Lp$-sentence $\ph$,  $\abs{\ph}=\abr{\ph}$, i.e.\ for all $r\in\R$, 
$\M^*,r\models\ph$ iff $\ph\in\D_r$.
\end{theorem}

\begin{proof}
By induction on the number of connectives  and quantifiers of $\ph$. 

If $\ph$ is an atomic sentence $P\tau_1\cdots\tau_n$, then the $\tau_i$ are closed terms and
 the semantics gives 
$$
\abs{\ph}=\abs{P}(\abs{\tau_1},\dots,\abs{\tau_n}).
$$
Putting $a_i=\abs{\tau_i}$, the definition of 
$\abs{P}$ yields 
$\abs{\ph}= \abr{P\ov{a_1}\cdots\ov{a_n}}$.

Now the models $\M$ and $\M^*$ agree on all constant terms, so $\abm{\ov{a_i}}=a_i=\abm{\tau_i}$. Hence the sentence
$$
P\ov{a_1}\cdots\ov{a_n} \leftrightarrow P\tau_1\cdots\tau_n
$$
is true throughout $\M$, so belongs to every $\D_r$ by Lemma \ref{phiR}.
Therefore
$\abr{P\ov{a_1}\cdots\ov{a_n}} =\abr{P\tau_1\cdots\tau_n}$, giving
$\abs{P\tau_1\cdots\tau_n} =\abr{P\tau_1\cdots\tau_n}$ as required.

The other base case is when the sentence $\ph$ is an identity $\tau_1\id\tau_2$. Then we either have  
$\abs{\tau_1}=\abs{\tau_2}$ and $\abs{\tau_1\id\tau_2}=\R$, or 
$\abs{\tau_1}\ne\abs{\tau_2}$ and $\abs{\tau_1\id\tau_2}=\emp$.
Since $\abs{\tau_i}=\abm{\tau_i}$, the first option gives  $\abm{\tau_1}=\abm{\tau_2}$, hence $\M\models \tau_1\id\tau_2$, and so $\abr{\tau_1\id\tau_2}=\R$ by Lemma \ref{phiR}.
The second option gives $\abm{\tau_1}\ne\abm{\tau_2}$, hence $\M\models \neg(\tau_1\id\tau_2)$, 
so  $\abr{\neg(\tau_1\id\tau_2)}=\R$, and thus $\abr{\tau_1\id\tau_2}=\emp$.
Both options have $\abs{\tau_1\id\tau_2}=\abr{\tau_1\id\tau_2}$ as required.

The inductive cases for the Boolean connectives are standard by properties of maximal sets, e.g.\ assuming the result for $\ph$ we have
$$
\abs{\neg\ph}=\R-\abs{\ph}=  \R-\abr{\ph}=\abr{\neg\ph}.
$$

For the temporal modalities we have

$$
\begin{array}{ll@{\qquad}l}
\abs{\b G\ph} &= [<] \abs{\ph} &\text{semantics of }\b G
\\
&=  [<] \abr{\ph}  &\text{induction hypothesis}
\\
&= \abr{\b G\ph}  &\text{Lemma \ref{GHsem}},
\end{array}
$$
and similarly for $\b H$. For the quantifiers we have

$$
\begin{array}{ll@{\qquad}l}
\abs{\forall x\ph} &=  \bscap_{a\in U}  \abs{\ph(\bar{a} /x)}   &\text{semantics of  $\forall$ and Substitution \eqref{termsub} }
\\
&= \bscap_{a\in U} \abr{\ph(\bar{a}/x)}  &\text{induction hypothesis}
\\
&= \abr{\forall x\ph} &\text{Lemma \ref{allchar}.}
\end{array}
$$
\end{proof}
\begin{corollary}
$\M^*$ is an $\Lp$-model, i.e.\ any $\Lp$-formula $\ph$ is admissible in $\M$. 
\end{corollary}

\begin{proof}
For any variable assignment $f:\Var\to U$, by the version of Lemma \ref{substf} that holds for $\M^*$ we have $\abs{\ph}f=\abs{\ph^f}=\abr{\ph^f}\in\Prop$, as required.
\end{proof}

\begin{corollary}
\begin{enumerate}[\rm(1)]
\item
For any sentence $\ph$ and $q\in\BBQ$, \enspace
$\M^*,q\models\ph$ iff $\M,q\models\ph$.
\item
For any formula $\ph$,  $q\in\BBQ$, and $f:\Var\to U$, \enspace
$\M^*,q,f\models\ph$ iff $\M,q,f\models\ph$.
\end{enumerate}
\end{corollary}

\begin{proof}
(1) This follows from the Theorem, as  $\ph\in\D_q$ iff $\M,q\models\ph$.

(2) This follows from part (1) and Lemma \ref{substf} for $\M$ and $\M^*$, using the sentence $\ph^f$.
\end{proof}

From the last result and the fact that $\M,q_0,f_0\models\D^+$ we infer that $\M^*,q_0,f_0\models\D^+$, hence
$\M', q_0,f_0\models\D$, where $\M'$ is the $\c L$-reduct of $\M^*$. That concludes the proof of strong completeness of $\el_\R$ for admissible $\c L$-models over $(\R,<)$.

The model $\M^*$ need not be Kripkean. Indeed there must be an $\el_\R$-consistent formula $\ph$ that is not satisfiable in any Kripkean model over $\R$. Otherwise, $\el_\R$ would be complete for validity in such models, contrary to Scott's non-axiomatisability result. So an $\M^*$ satisfying this $\ph$, as produced by the above construction, must be non-Kripkean.

The Barcan formulas \eqref{barcan} are valid in Kripkean models, and play an essential role in the completeness proof for $\el_\Q$ of Theorem \ref{LQcomplete}, in showing that certain sets are $\forall$-complete, which enables a Kripkean model to be constructed \cite[p.121]{gabb:temp94}. But $\M^*$ validates $\el_\R$, since it is an admissible model over $(\R,<)$, and hence it validates the Barcan formulas, since they are $\el_\R$-theorems, even though $\M^*$ is not in general Kripkean.

The $\M^*$ construction only requires $\D_r$ to be maximally consistent, not $\forall$-complete as would be required to get a Kripkean model. 
While the Kripkean condition is sufficient for validity of the Barcan formulas, we see now that it is not necessary.

\section{Non-axiomatisability over $\Z$}

Scott's work on non-axiomatisability of the temporal logic of standard models is not published. A detailed treatment of the topic is given in \cite[\S2.8]{gabb:temp94}. We now analyse this further to show that the temporal logic of admissible models over $\Z$ is not recursively axiomatisable.

Let $\La=\{0,{}',+,\times\}$ be the signature for the first-order language of arithmetic, with identity. Let $\c N=(\om,0,{}',+,\times)$ be  \emph{the standard model of arithmetic}, comprising the set of natural numbers on which $0,{}',+,\times$ have their standard arithmetical meanings.

Let $\c L=\La \cup\{e,q,\prec\}$ where $e$ and $q$ are unary predicate symbols and $\prec$ is a binary predicate symbol.
We will be concerned with admissible models for this signature based on the integer time flow, i.e.\ models of the form
\begin{equation}   \label{modelM}
\M=(\Z,<,\Prop,U,\abm{\dsh}).
\end{equation}
The symbols of $\La$ are interpreted as operations on the universe $U$. We will abbreviate the interpretation $\abm{0}$ of $0$ to $0^\M$, and write the interpretations $\abm{'}$, $\abm{+}$, $\abm{\times}$  of the other function symbols just as ${}',+,\times$, allowing the context to determine what is intended.
The role of the predicate $e$ will be to provide an embedding of $U$ into $\Z$ by associating with each $a\in U$ a unique time in $\Z$ at which $a$ satisfies $e$. The symbol $\prec$ will rigidly define an ordering on $U$ that matches $<$ under this embedding.  The role of $q$ will be to single out a subset $U_q$ of $U$ that is closed under the operations interpreting $\La$ and forms an isomorphic copy of the standard model $\c N$. The  interpretation $\abm{q}:U\to\Prop$ of $q$  will also be rigid, so  $\abm{q}(a)$ is either $\Z$ or $\emp$.

Let $\mu$ be the conjunction of the following sentences:
\begin{enumerate}[(i)]
\item
$\forall x\lozenge \big(e(x)\land\b G\neg e(x)   \land \b H\neg e(x)\big)$
\item
$\Box\forall x\forall y(e(x)\land e(y)\to x\id y)$
\item
$\Box\forall x\forall y\big( x\prec y \leftrightarrow \lozenge(e(x)\land\b F e(y)\big)$
\item
$\Box\forall x(q(x)\to\Box q(x))$
\item
$\Box[ q(0)\land \forall y (y\prec 0 \to\neg q(y)) ]$
\item
$\Box\forall x[q(x) \to (x\prec x'\land q(x')\land\forall z(x\prec z\land z\prec x' \to \neg q(z)))]$
\item
$\Box \forall x\forall y[q(x)\land q(y) \to q(x+y)\land q(x\times y)]$
\item
$\Box\forall x(q(x) \to x+0\id x)$
\item
$\Box \forall x\forall y[q(x)\land q(y) \to x+y'\id(x+y)' ]$
\item
$\Box\forall x(q(x) \to x\times 0\id 0)$
\item
$\Box \forall x\forall y[q(x)\land q(y) \to x\times y'\id(x\times y)+x]$
\end{enumerate}

These correspond to the sentences (1)--(11) in \cite[p130]{gabb:temp94}, except for (ii) which replaces the stronger
$$
\Box\exists x (e(x)\land \forall y(x\not \id y \to \neg e(y)))
$$
used in that reference.
It is significant that $\mu$ as defined here contains no existential quantifiers, and its occurrences of $\forall$ allow us to apply the principle that

\begin{equation}\tag{2}
\text{
$\M,t,f\models\forall x\ph$  {implies }   $\M,t,f[a/x]\models\ph$  for all $a\in U$,}
\end{equation}
which holds in all admissible models (but its converse may not). 
For instance, if sentence (iv) holds at some time in $\M$, then the sentence 
$\forall x(q(x)\to\Box q(x))$ holds at all times, so from \eqref{allkrip}, for each $a\in U$, if  $a$ satisfies $q$ at some $t$, i.e.\ $t\in\abm{q}(a)$, then  $a$ satisfies $q$ at every time, hence $\abm{q}(a)=\Z$. Otherwise $\abm{q}(a)=\emp$.
Thus $q$ is interpreted rigidly in $\M$.
 In what follows, the use of sentences (i)--(xi) all depend in this sort of way on \eqref{allkrip} but not its converse.
 
 Let $U_q=\{a\in U: \abm{q}(a)\ne\emp\}$. We now show that the sentence $\mu$ forces $U_q$ to be a copy in $\M$ of the standard model of arithmetic.
 
\begin{theorem}  \label{thone}
Let $\M$ be an admissible $\c L$-model over integer time as in \eqref{modelM}.  If the sentence $\mu$ holds at some point of $\Z$ in $\M$, then $U_q$ is closed under the operations interpreting $0,{}',+,\times$ and forms an $\La$-model $\c U_q$ isomorphic to $\c N$.
\end{theorem}

\begin{proof}
Assume $\mu$ holds at some point in $\M$, hence each of its conjuncts does. Applying principle \eqref{allkrip} to sentence (i) we get that for each $a\in U$ there is exactly one time $t\in \Z$ at which $a$ satisfies $e$. Put $\th(a)=t$. This defines a function $\th:U\to\Z$. By sentence (ii), $\th$ is injective.

By sentence (iii), for any $a,b\in U$ and any time $t\in\Z$, we have $t\in\abm{{\prec}}(a,b)$ iff $\th(a)<\th(b)$.
So $\prec$ is interpreted rigidly in $\M$ and defines a binary relation $\pm$ on $U$ by putting $a\pm b$ iff $\abm{{\prec}}(a,b)\ne\emp$ iff $\abm{{\prec}}(a,b)=\Z$, 
making $a\prec^\M b$ iff $\th(a)<\th(b)$. It follows that $\pm$ inherits many properties of the ordering $<$, including transitivity, irreflexivity and linearity.

As noted above,  (iv) ensures that $q$ is interpreted rigidly, so $U_q=\{a\in U: \abm{q}(a)=\Z\}$.
 Sentence (v) ensures that $0^\M$ belongs to $U_q$, while sentences (vi) and (vii) yield that $U_q$ is closed under the operations   ${}',+,\times$, so forms an $\La$-structure $\c U_q$.

A sequence $\{a_n:n<\om\}$ of distinct elements of $U_q$ can be defined inductively by putting $a_0=0^\M$ and $a_{n+1}=a_n{}'$. By (vi), $a_n\pm a_{n+1}$, so we have a strictly increasing sequence
$$
a_0\pm a_1\pm\cdots \cdots a_n\pm  a_{n+1}\pm \cdots\cdots\cdots
$$
Now we show that $U_q=\{a_n:n<\om\}$. First,
(v) ensures that $a_0$ is the $\pm$-least member of $U_q$, and (vi) ensures that there is no member of $U_q$ that is $\pm$-between $a_n$ and $a_{n+1}$ for any $n<\om$. Therefore if there exists some $b\in U_q$ with $b\ne a_n$ for all $n<\om$, then we must have $a_n\pm b$ for all $n$. Applying the injective order-preserving $\th$ then gives  
$\th(a_0)< \th(a_{n})<\th(b)$ for all $n>0$. But this is impossible, as there are infinitely many $\th(a_n)$'s, but only finitely many integers between $\th(a_0)$ and $\th(b)$.
 So we  conclude that $b$ cannot exist, and therefore 
$U_q=\{a_n:n<\om\}$.

The sentences (viii)--(xi) can be used to show, by induction on $n$,  that in general $a_{m+n}=a_m+a_n$ and 
$a_{m\times n}=a_m\times a_n$. Thus the map $n\mapsto a_n$ is an isomorphism from $\c N$ onto $\c U_q=(U_q,0^\M,{}',+,\times)$.
\end{proof}

To discuss the language of arithmetic within $\M$ we relativise the quantifiers of $\La$-formulas to range over $U_q$. A translation is inductively defined, taking each $\La$-formula $\ph$ to an $\La\cup\{q\}$-formula $\ph_q$, by putting 
$\ph_q=\ph$ if $\ph$ is atomic; letting the map $\ph\mapsto\ph_q$ commute with the Boolean connectives, i.e.\ 
$(\neg\ph)_q=\neg(\ph_q)$,  $(\ph\land\psi)_q=\ph_q\land\psi_q$ etc.; and  $(\forall x\ph)_q=\forall x(q(x)\to\ph_q)$.

\begin{lemma}    \label{lemtwo}
Let $\M=(T,<,\Prop,U,\abm{\dsh})$ be any admissible $\c L$--model in which $q$ is interpreted rigidly and  $U_q$ is a subalgebra of $(U,0^\M,{}',+,\times)$. Let $\ph$ be any $\La$-formula. Then for any variable assignment $f:\Var\to U_q$ and any $t\in T$,  we have $\c U_q,f\models\ph$ iff $\M,t,f\models \ph_q$.
\end{lemma}

\begin{proof}
Since $U_q$ is closed under $0,{}',+,\times$, it forms an $\La$-structure $\c U_q$ in which any interpretation $\ab{\tau}^{\c U_q}f$ of any $\La$-term is identical to $\abm{\tau}f$. The notation  $\c U_q,f\models\ph$ expresses the classical (non-modal) satisfaction relation in this structure. In particular  
\begin{equation} \label{forallUq}
\text{$\c U_q,f\models\forall x\ph$ iff  every $a\in U_q$ has   $\c U_q,f[a/x]\models\ph$ .}
\end{equation}
The statement of the lemma implies that the formulas  $\ph_q$ are interpreted rigidly in $\M$: if $\M,t,f\models \ph_q$ holds for some $t$ then it holds for all.

We prove the lemma by induction on the formation of $\La$-formulas, which are constructed from atomic formulas by Boolean connectives and $\forall$. If $\ph$ is atomic, then it is an equation $\tau_1\id\tau_2$, and is equal to $\ph_q$. We have
$\c U_q,f\models\tau_1\id\tau_2$  iff  $\ab{\tau_1}^{\c U_q}f=\ab{\tau_2}^{\c U_q}f$ iff  $\ab{\tau_1}^{\M}f=\ab{\tau_2}^{\M}f$, which is precisely the condition for  $\M,t,f\models\tau_1\id\tau_2 $ to hold for any $t\in T$. Hence the lemma holds for atomic formulas.

The inductive cases of the Boolean connectives are straightforward. For the case of $\forall$, assume inductively that the result holds for $\ph$. Suppose $\c U_q,f\models\forall x\ph$.  For any $t\in T$ and $a\in U$, if
$\M,t,f[a/x]\models q(x)$, then $t\in\abm{q}(a)$ so $a\in U_q$, hence  $\c U_q,f[a/x]\models\ph$  by \eqref{forallUq}, therefore $\M,t,f[a/x]\models \ph_q$ by the induction hypothesis on $\ph$. 

This shows that  $\M,t,f[a/x]\models q(x)\to\ph_q$ for every $t\in T$ and $a\in U$. Hence
\begin{equation}
\bigcap\nolimits_{a\in U}\abm{ q(x)\to\ph_q}f[a/x]=T.
\end{equation}
But $T\in\Prop$, so any $t\in T$ belongs to a member of $\Prop$ that is included in the left side of this last equation. By the admissible semantics of $\forall$, this means that $\M,t,f\models\forall x(q(x)\to\ph_q)$, i.e.\ 
$\M,t,f\models (\forall x \ph)_q$.

Conversely, suppose  $\M,t,f\models (\forall x \ph)_q$. Then each $a\in U_q$ has $\M,t,f[a/x]\models  q(x)\to\ph_q$ (see \eqref{allkrip}), and $\M,t,f[a/x]\models  q(x)$ as $\abm{q}(a)=\Z$, so then $\M,t,f[a/x]\models \ph_q$. Hence  $\c U_q,f[a/x]\models\ph$ by the induction hypothesis on $\ph$. This shows $\c U_q,f\models\forall x\ph$ by  \eqref{forallUq}.

Altogether now we have shown that the lemma holds for $\forall x\ph$, completing the inductive case of $\forall$, and hence the inductive proof for all formulas.
\end{proof}

\begin{theorem}  \label{trueinN}
An $\La$-sentence $\ph$ is true in $\c N$ iff the $\c L$-sentence $\mu\to\ph_q$ is valid in all admissible $\c L$-models over $(\Z,<)$.
\end{theorem}

\begin{proof}
Suppose $\c N\models\ph$. If $\M$ is admissible over $(\Z,<)$ and $\M,t\models \mu$, then by Theorem \ref{thone} the $\La$-structure $\c U_q$ within $\M$ is isomorphic to $\c N$, hence $\c U_q\models \ph$, so $\M,t\models \ph_q$ by Lemma \ref{lemtwo}.  This shows  $\M,t\models \mu\to\ph_q$ for any $t\in\Z$, i.e. $\mu\to\ph_q$ is valid in $\M$.

Conversely let $\mu\to\ph_q$ be valid in all admissible models over the time flow $(\Z,<)$. Define such a model
$\M$ by putting $\Prop=\wp\Z$, the full powerset of $\Z$, and $U=\Z$, with the symbols of $\La$ having their standard interpretations in $\Z$. Interpret $e$ in $\M$ by defining $\abm{e}(a)=\{a\}\in\Prop$ for all $a\in U$. For $q$  define $\abm{q}(a)$ to be $\Z$ if $a\in\om$, and $\emp$ otherwise. Then $q$ is interpreted rigidly in $\M$, and
$U_q=\om$, so $U_q$ is closed under the $\La$-operations.
For $\prec$  define $\abm{{\prec}}(a,b)$ to be $\Z$ if $a<b$, and $\emp$ otherwise. Then the relation $\pm$ on $U (=\Z)$ is  $<$.  

Taking any $t\in\Z$, we have that $\M,t\models\mu$, with 
  $\th:U\to\Z$ being the identity function on $\Z$. 
 By the assumed validity of $\mu\to\ph_q$ we get  $\M,t\models\ph_q$, so by Lemma \ref{lemtwo},
 $\c U_q\models\ph$. But the $\La$-structure $\c U_q$ in this case is  just $\c N$,
 so $\c N\models\ph$ as required.
\end{proof}

\begin{corollary}
 The set of formulas valid in all admissible $\c L$-models over integer time is not recursively axiomatisable. 
\end{corollary}
\begin{proof}
This is a  well known argument. Any recursively axiomatisable logic has a recursively enumerable set of theorems \cite[Theorem 7.1]{ ende:elem77}. But from a recursive enumeration of the formulas valid in all admissible $\c L$-models over $\Z$ we could obtain via Theorem \ref{trueinN} a recursive enumeration of the true sentences of the standard model of arithmetic, something that does not exist.
\end{proof}

In Theorem \ref{thone} the negative members of $\Z$ do not play a particular role and can be dispensed with. The theorem remains true if $\Z$ is replaced by $\om$ as the time flow, leading to a proof that the logic of formulas valid in admissible models over natural number time is not recursively axiomatisable. 

 This non-axiomatisability of logics over $\Z$ or $\om$ is covered by the more general fact that it holds for the temporal logic of admissible models over any \emph{scattered} linear order, which is one that does not contain any dense suborder, i.e.\ one into which $\Q$ cannot be order-embedded.   Hausdorff showed that if $\c O$ is the class of all well-orderings and their inverses, then the scattered linear orders form the smallest class of linear orderings that includes $\c O$ and is closed under lexicographical sums indexed by a member of $\c O$ \cite[\S5.3]{rose:line82}.
  
 If the $\c L$-sentence $\mu$ is modified by adding the sentences (12)--(15) of \cite[p131]{gabb:temp94} as conjuncts, then by the analysis given in that reference, if the sequence $\{a_n:n<\om\}$ in the proof of our Theorem  \ref{thone} did not exhaust $U_q$, then there would exist an order-embedding $\eta:(\Q\,<)\to (U_q,\prec^\M)$. The modified $\mu$ has a single subformula of the form $\exists z\psi$, but the formula $\psi$ there is rigid, so the existential quantifier gets the Kripkean interpretation \eqref{kripsemExists}. Hence the analysis of  \cite{gabb:temp94} for a standard model holds also for an admissible one. Then composing $\eta$ with the map $\theta$ from the proof of Theorem \ref{thone} would give an order-embedding of $\Q$ into the time flow of $\M$, showing that the latter was not scattered. Accordingly, if the time flow \emph{is} scattered, then  $U_q=\{a_n:n<\om\}$.
 This leads to a proof that the logic of formulas valid in admissible models over a scattered time flow is not recursively axiomatisable. 
 
 As a final topic, noting that our results imply that Theorem \ref{thone} must fail in general for models over the real time flow,
 we  construct an $\c L$-model $\M=(\R,<,\Prop,U,\abm{\dsh})$ that exhibits this failure. $\M$  satisfies the sentence $\mu$ but does not have $\c U_q$ isomorphic to the standard model of arithmetic $\c N$.

Let $\Prop=\wp\R$. 
Take $\c U=(U,0,{}',+,\times)$ to be a countable nonstandard model of arithmetic, say a proper elementary $\La$-extension of $\c N$. Interpret $q$ to hold rigidly in $\M$ of every member of this structure i.e.\ $\abm{q}(a)=\R$ for all $a\in U$. Then $\c U_q$ is just $\c U$, which is not isomorphic to $\c N$.  Interpret $\prec$ rigidly to be the ordering relation of $\c U$.

To interpret $e$, suppose temporarily that we have an injective order-preserving function $\th$ from $U$ to $\R$. Then we use it to define $\abm{e}(a)=\{\th(a)\}$ for all $a\in U$, completing the definition of $\M$. It can then be seen that $\M,t\models\mu$ for any $t\in\R$.

So it remains to show that a $\th$ as described does exist. We use the well-known description of  non-standard models of arithmetic \cite[Section 3.1.]{robi:stan66}: $\c U$ comprises a standard part, a copy of $\c N$, followed linearly by countably many pairwise disjoint copies of $(\Z,{}')$, called \emph{galaxies}. A typical galaxy looks like
$$
\cdots\cdots\prec^\M\bullet \prec^\M a\prec^\M a'\prec^\M a''\prec^\M\cdots\cdots
$$
The set of galaxies is linearly ordered by declaring one galaxy to be less than another if every member of the first is $\prec^\M$-less than every member of the second in $\c U$. There is no least galaxy in this ordering and no greatest, while between any two galaxies there lies a third.
So the set of galaxies forms  a countable dense linear ordering without endpoints, and hence is isomorphic to the ordered set $(\Q,<)$ of rationals, by a celebrated theorem of Cantor \cite[Theorem 2.8]{rose:line82}.

Any galaxy $\varGamma$ can be embedded into any open interval $I$ of the real line. Let  $\varGamma=\{a_j:j\in\Z\}$ with $a_j\prec^\M a_k$ iff $j<k$. Use the density of the reals to find a subset $\{t_j:j\in\Z\}$ of $I$ with $t_j < t_k$ iff $j<k$. Then the map $a_j\mapsto t_j$ is an order preserving embedding of $(\varGamma,\prec^\M)$ into $(I,<)$.

To define $\th$ we need a countable collection $\c I$ of pairwise disjoint open intervals of $\R$  that we can match bijectively  with the galaxies, and with $\c I$ having the same order type as $(\Q,<)$ under the linear ordering $\lessdot$ that puts $(s,t)\lessdot(s',t')$ iff $(s,t)$ is strictly to the left of  $(s',t')$, i.e.\ iff $t<s'$.
We can find such a $\c I$ within any given open interval $(r,s)$ on the real line.
As a first step choose some proper subinterval $(r_1,s_1)$ of $(r,s)$ and put it into $\c I$.
$$
r  \text{ --------------------------- } r_1 \cdots\cdots s_1 \text{ --------------------------- }   s
$$
This leaves two  `open pieces' $(r,r_1)$ and $(s_1,s)$ within $(r,s)$. At step two choose proper subintervals $(r_{21},s_{21})$  and
 $(r_{22},s_{22})$ of $(r,r_1)$ and $(s_1,s)$, respectively, and put them into $\c I$. 
$$
r  \text{ ------ } r_{21} \cdots s_{21}     \text{ ------ }      r_1 \cdots\cdots s_1 \text{ ------ } r_{22}\cdots s_{22} \text{ ------ } s.
$$
That leaves four open pieces within $(r,s)$ to have proper subintervals chosen. Iterating these steps  countably many times gives the desired $\c I$ whose  ordering $\lessdot$ is dense and without end-points.

For the definition of $\th$, start with the standard part of $\c U$, which can be identified as an ordered set with $(\om,<)$, and put $\th(n)=\frac{n}{n+1}$ for all $n\in\om$ to give an order-preserving injection of $(\om,<)$ into the interval $[0,1)$ of $\R$.
Then take any open interval $(r,s)$ with $1\leq r$ and let $(\c I,\lessdot)$ be the ordered set of open subintervals of $(r,s)$ constructed above. There is an order-isomorphism $\eta$ from the set of galaxies onto $(\c I,\lessdot)$, because both structures have the same order type as $(\Q,<)$. Extend $\th$ to act on each galaxy $\varGamma$ as an order preserving embedding of $(\varGamma,\prec^\M)$ into $(\eta(\varGamma),<)$, as described above. That completes the construction of $\th$ as an order-preserving injection of $(U,\prec^\M)$ into the real line, and hence completes the counter-example to Theorem \ref{thone} over $\R$.

\subsection*{Acknowledgment} 
 
 I thank Ian Hodkinson for helpful discussions and suggestions that improved the presentation, including extending the non-axiomatisability result for integer time to scattered orders.

\bibliographystyle{plain}

 \end{document}